\documentclass[ a4paper, 11pt]{article}

\usepackage[T1]{fontenc}

\usepackage{amssymb}

\usepackage{amsmath}
\usepackage{amsthm}
\usepackage[francais,english]{babel}
\usepackage{graphicx}
\usepackage{epsfig}
\usepackage{color}
\usepackage{latexsym}
\usepackage[all]{xy}
\usepackage{geometry}
\usepackage{caption}

\newcommand\rot{\mathrm{rot}}
\newcommand\seg{{\mathrm{seg}}}
\newcommand\Id{\mathrm{Id}}
\newcommand\Kh{\mathrm{Kh}}

\newcommand\loc{\mathrm{loc}}
\newcommand\rel{\mathrm{rel}}

\newtheorem*{claim}{Claim}

\newtheorem*{conj}{Conjecture}

\newtheorem{thm}{Theorem}

\newtheorem{cor}[thm]{Corollary}

\newtheorem{lem}[thm]{Lemma}

\title{Grid diagrams and Khovanov homology}

\author{Jean-Marie Droz and Emmanuel Wagner}

\date{ }

\begin{document}

\maketitle

\abstract{We explain how to compute the Jones polynomial of a link from one of its grid diagrams and we observe a connection between Bigelow's homological definition of the Jones polynomial and Kauffman's definition of the Jones Polynomial. Consequently, we prove that the Maslov grading on the Seidel-Smith symplectic link invariant coincides with the difference between the homological grading on Khovanov homology and the Jones grading on Khovanov homology. We give some evidence for the truth of the Seidel-Smith conjecture.}

\section*{Introduction}

 Using symplectic geometry, Seidel and Smith constructed an invariant of oriented links in $S^3$ \cite{SS}. The Seidel-Smith invariant of an oriented link $L$ is defined as the homology $\Kh_{symp}^*(L)$ of a chain complex associated to $L$. The homological grading of this chain complex is denoted by $P$. They also conjectured that this invariant is isomorphic to the Khovanov link homology $\Kh^{*,*}(L)$:\\

\begin{conj}[Seidel and Smith]
For all $k \in \mathbb{Z}$, 
$$\Kh_{symp}^k(L)\cong \bigoplus_{\begin{array}{c}(i,j)\in \mathbb{Z}^2\\ i-j=k\end{array}}\Kh^{i,j}(L),$$
where $i$ is called the Khovanov homological grading and $j$ the quantum grading.
\end{conj}

Manolescu showed \cite{M} that the generators of the Seidel-Smith chain complex are in one-to-one correspondence with the intersection points between homology representative arising in Bigelow's construction of the Jones polynomial \cite{Bi}. Moreover, this correspondence allows us to endow the Seidel-Smith generators with Bigelow's Jones grading $J$. Supporting the Seidel-Smith conjecture, it has been observed by Manolescu that there is, on small examples, enough generators in the Bigelow construction to have a complex generated by them whose homology coinsides with the Khovanov homology (taking into account the gradings).\\

The purpose of the present article is to introduce a differential on the graded free abelian group generated by the Bigelow intersection points, also called Bigelow's generators. This aim is achieved by proving that there is an injection of Bigelow's generators into enhanced Kauffman states \cite{Kh, Vi}. In other words, we will see the Seidel-Smith generators as a subset of the generators of the Khovanov chain complex. Moreover, we prove that, as expected, the gradings verify $P=i-j$ and $j=J$. Our main theorem is:


\begin{thm}\label{main}
There exists  a differential $\delta$ on the free abelian group $B$ generated by Bigelow's generators, that respects $J$, increases $P$ by 1 and such that the homology of the chain complex $(B,\delta)$ is the Khovanov homology.
\end{thm}

Our complex $(B,\delta)$ is homotopic to the original combinatorially defined complex of Khovanov \cite{Kh}. This gives strong evidence supporting the Seidel-Smith conjecture by generalizing Manolescu's observation. Our main theorem remains true for odd Khovanov homology \cite{ORS}, see Theorem \ref{odd}.\\

The techniques used below are of intrinsic interest. Namely, we develop a combinatorial description of the Jones polynomial and of Khovanov homology in terms of {\it rectangular diagrams} (see Section \ref{rect}). In particular, our result gives an alternative proof of the 
equivalence of Bigelow's definition of the Jones polynomial. In addition, grid diagrams appear in the combinatorial description of  link Floer homology \cite{MOST}. This new description will be used in future work to investigate the relation between Khovanov type homologies and Heegaard-Floer type homologies (see e.g. \cite{OS2, RKH, GW}).\\

{\bf Plan of the paper.} In the first section, we introduce all definitions and notations necessary to compute the Jones polynomial from a grid diagram. In the second section, we construct an injection of Bigelow's generators into enhanced Kauffman's states. In the third section, we prove relations between the gradings. The last section is devoted to the proof of the main theorem.\\

{\bf Aknowledgements. }
We would like to thank Anna Beliakova for helpful conversations and pointing to us the paper of Manolescu. The idea of using rectangular diagrams for the Bigelow setting was suggested to us by Anna Beliakova. The present article would not exist without the kind encouragements of Benjamin Audoux.\\

\section{Definitions and notations}\label{rect}

\noindent {\bf Grid diagrams and links.}
A {\it grid diagram} of size $n\in \mathbb{N}-\{0,1\}$ is a $(n\times n )-$grid whose squares may be decorated by either an $O$ or an $X$ so that each column and each row contains exactly one $O$ and one $X$. The number $n$ is called the {\it complexity} of the grid diagram. Following \cite{MOST}, we denote by $\mathbb{O}$ the set of $O$'s and $\mathbb{X}$ the set of $X$'s. The $X$'s and the $O$'s are called the {\it punctures} of the grid diagram.

From any grid diagram, one can construct an oriented link diagram. For this purpose, one should join the $X$ to the $O$ in each column by a vertical segment and the $O$ to the $X$ in each row by an horizontal segment that passes under all the vertical segments. We choose the orientation to be from the $O$'s to the $X$'s on the horizontal lines and from the  $X$'s to the $O$'s on the vertical lines. This produces a {\it planar rectangular diagram} for an oriented link in $S^3$. Any oriented link in $S^3$ admits a planar rectangular diagram \cite{Dy}. An example is shown in Figure \ref{grid}.\\
\begin{figure}[!h]
\begin{center}
\includegraphics{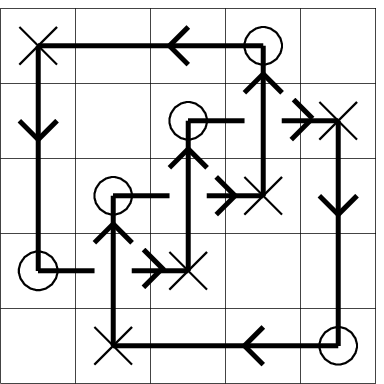}
\caption{Grid diagram and planar rectangular diagram for the trefoil}
\label{grid}
\end{center}
\end{figure}

\noindent {\bf The set $\mathcal{G}$ of the Bigelow generators.}
 Let $D$ be a planar rectangular diagram of complexity $n$. For each vertical segment, let us draw an oriented curve which winds clockwise around the puncture at the top of the segment and counterclockwise around the puncture at the bottom of the segment. The oriented curves obtained are called {\it figure-eights}. We choose the figure-eights very narrow and as short as possible. We assume that the figure-eights intersect transversally and at most twice the horizontal segments, do not intersect each other and have only one transversal self-intersection. 
We denote by $\mathcal{G}$ the set of unordered $n$-tuples of intersection points between horizontal segments and (vertical) figure-eights such that each (vertical) figure-eight and each (horizontal) segment contains exactly one point.
 We denote by $\mathcal{Z}$ the set of intersection points between figure-eights and horizontal segments. We define ${\bf x}\in\mathcal{G}$ (${\bf o}\in\mathcal{G}$, respectively) as the set of points of $\mathcal{Z}$ that are nearest to the $X$'s (the $O$'s, respectively), see Figure \ref{figureight} for an example. To each element $g \in \mathcal{G}$, one can associate a unique  $n$-tuple, $\overline{g}=(\overline{g_1},\ldots,\overline{g_n})$ in which $\overline{g_i}$ ($i=1,\ldots,n$) is the $X$, $O$ or crossing nearest to $g_i$.\\
\begin{figure}[!h]
\begin{center}
\includegraphics{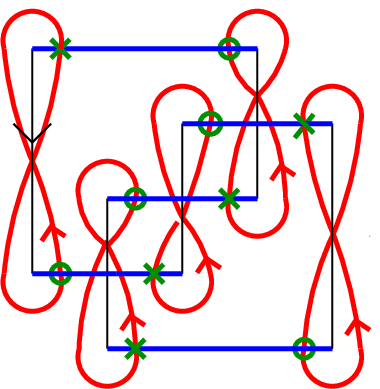}
\caption{Oriented figure-eights and particular elements {\bf x} and {\bf o} of $\mathcal{G}$ }
\label{figureight}
\end{center}
\end{figure}

\noindent {\bf Gradings on $\mathcal{G}$.}
We will define three gradings on $\mathcal{G}$: $\overline{P}$, $T$ and $Q$. For the first grading $\overline{P}:\mathcal{G}\rightarrow \mathbb{Z}$, let us orient figure-eights as in Figure \ref{figureight}.  Each $z\in\mathcal{Z}$ is given an integer $p(z)$: $+1$ if the part of the figure-eight on which $z$ sits is oriented upward, $0$ otherwise. Given $g=(z_1,\ldots, z_n)\in\mathcal{G}$, $\overline{P}(g)=p(z_1)+\dots+p(z_n)$. For example, in Figure \ref{figureight}, $\overline{P}({\bf x})=2 $ and $\overline{P}({\bf o})=2$.\\ 

Given finite sets of points $A$ and $B$ in the real plane, we define $\mathcal{I}(A,B)$ to be the number of pairs $(a_1,a_2)\in A$ and $(b_1,b_2)\in B$ such that $a_1<b_1$ and $a_2<b_2$. The grading $T:\mathcal{G}\rightarrow \mathbb{Z}$ is defined by $T(g)=\mathcal{I}(g,g)$, for $g\in \mathcal{G}$.\\

We define a relative grading $Q$ on $\mathcal{G}$. Consider two elements $g=(g_1,\ldots,g_n)$ and $h=(h_1,\ldots,h_n)$ in $\mathcal{G}$. To define the difference $Q(g)-Q(h)$, we consider the loop $\gamma(g,h)$ in the configuration space of $n$ points in $\mathbb{R}^2$ 
 defined as follows (see also \cite{Bi}). We start at $g$, go along the horizontal segments to $h$, then go back along the vertical figure-eights to $g$. We can also see $\gamma(g,h)$ as a family of closed immersed oriented curves in $\mathbb{R}^2$. Then $Q(g)-Q(h)$ is defined to be the sum of the winding numbers of these closed immersed curves around the $X$'s and the $O$'s. In other words, for each $X$ and $O$, we count algebraically the number of times each immersed curve goes around the puncture: $+1$ for each time a curve goes around a puncture counterclockwise and $-1$ for each time a curve goes around a puncture clockwise and we take the sum over all curves and all punctures as relative grading. We define the absolute grading $Q$ by setting $Q({\bf x})=0$.\\

\noindent {\bf Normalization of the gradings and the Jones polynomial.} We introduce two classical quantities associated to an oriented link diagram $D$ (and hence to an oriented planar rectangular diagram).
Given an oriented link diagram $D$, we resolve all the
crossings of $D$ as in Figure \ref{4},
\begin{figure}[!h]
\begin{center}
\includegraphics[height=1cm]{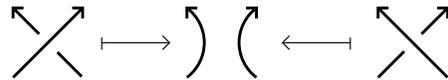}
\caption{Oriented resolution}
\label{4}
\end{center}
\end{figure}
we obtain a disjoint union of oriented circles embedded in $\mathbb{R}^2$. We call these circles the {\it Seifert circles} of $D$. 
The {\it rotation number} of $D$, denoted by $\rot(D)$, is the sum of the contributions of the circles. The contribution of a Seifert circle is $+1$ if it is oriented counterclockwise and $-1$ otherwise. Given a crossing $c$ of an oriented link diagram $D$, we define $w(c)$ as in Figure \ref{cross}. We define the {\it writhe} $w(D)$ of $D$, $$w(D)=\sum_{c\mathrm{\,crossings\, of\, } D} w(c).$$
We denote by $n_+$ the number of positive crossings and by $n_-$ the number of negative crossings of $D$. We have $w(D)=n_+-n_-$.

\begin{figure}[!h]
\begin{center}
\includegraphics[height=1cm]{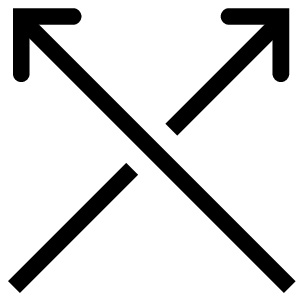}\hspace{5.6cm}\includegraphics[height=1cm]{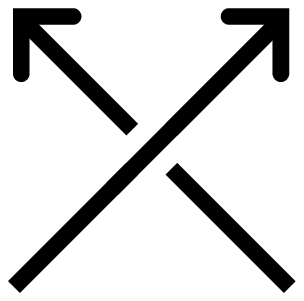}
\end{center}
\begin{center}
Negative crossing $c$, $w(c)=-1$\hspace{1.5cm}Positive crossing $c$, $w(c)=+1$
\end{center}
\caption{Crossings}
\label{cross}
\end{figure}

We define the Jones grading $J$ and the homological grading $P$. Given a planar rectangular diagram $D$, for an element $g\in \mathcal{G}$, we set
 $$J(g)=2(T(g)-Q(g))-2T({\bf x})+\rot(D)+w(D)\mbox{ and }$$
$$P(g)=\overline{P}(g)-\overline{P}({\bf x })-\rot(D)-w(D).$$

We prove below that the Jones polynomial of an oriented planar rectangular diagram $D$ can be written as:\\
\begin{equation}
V(D)(q)=\sum_{g \in \mathcal{G}} {(-1)}^{P(g)} q^{J(g)}.
\label{JFromB}
\end{equation}

\section{Bigelow's generators and enhanced Kauffman states}

\noindent {\bf From plat closures to rectangular diagrams.}
\begin{figure}[!h]
\begin{center}
\includegraphics[height=4cm,width=8cm]{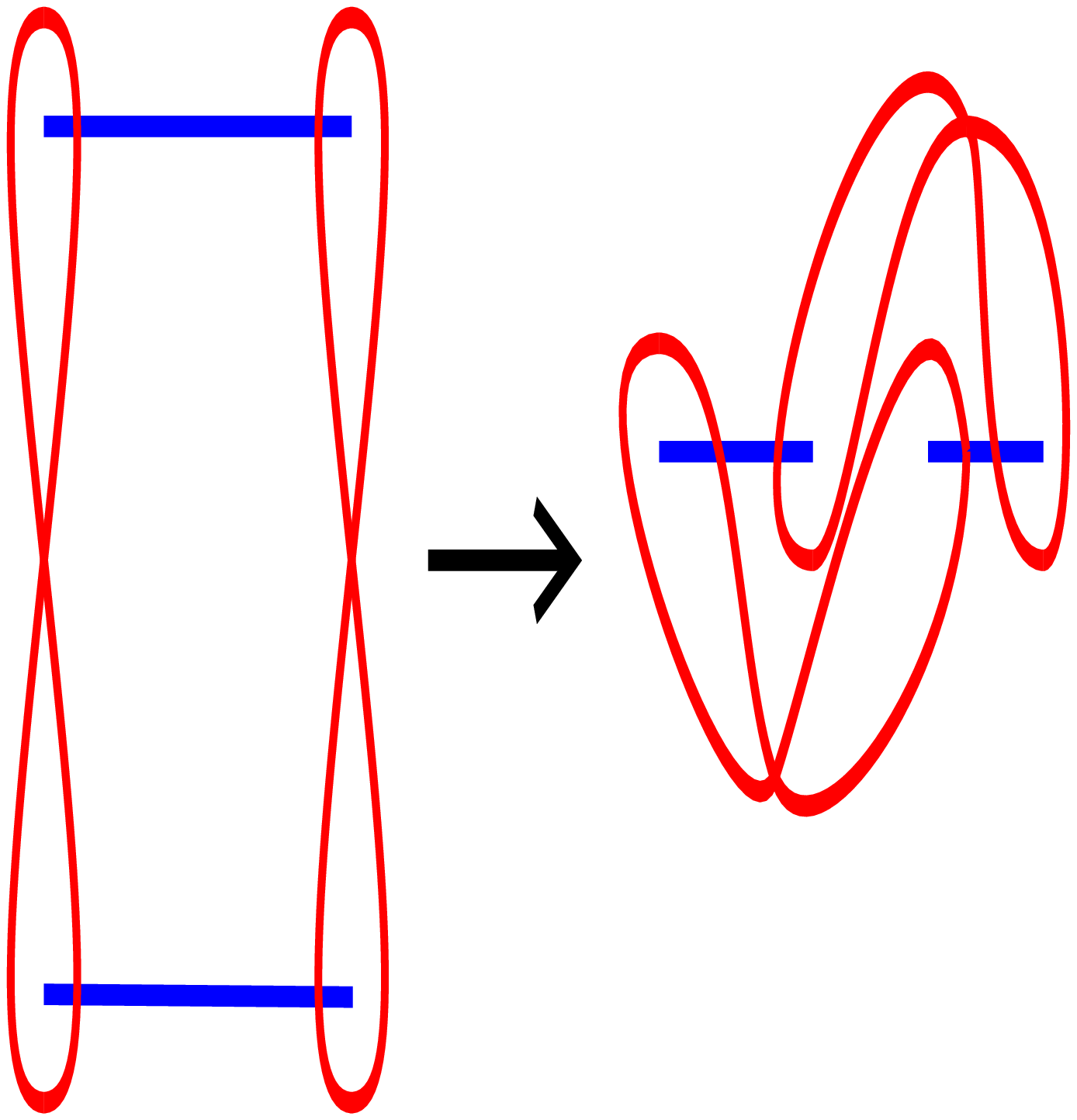}
\caption{From a rectangular diagram to a braid closure}
\label{bigTrans}
\end{center}
\end{figure}
Equation (\ref{JFromB}) is a simple reformulation of Bigelow's homological definition of the Jones polynomial. In \cite{Bi}, Bigelow computes the Jones polynomial of a link represented as the plat closure of a braid. Our set $\mathcal{G}$ is seen in the setting of Bigelow's definition as the set of intersection points between a cycle in homology represented by the figure-eights and another cycle represented by the horizontal segments.\\

In \cite{M} (Section 3, page 15 and 16), Manolescu explains how the plat closure of a braid can be given as a flattened braid diagram. A rectangular diagram can easily be transformed into a flattened braid diagram. Starting with a planar rectangular diagram in the plane with figure-eights drawn, we apply a diffeomorphism of the plane sending all horizontal segments of the rectangular diagram to consecutive non-intersecting segments on a line, see Figure \ref{bigTrans}. We obtain a flattened braid diagram. Our definitions of the different gradings are obtained by ``pulling back'' the gradings originally defined by Bigelow along the diffeomorphism defined above. Notice that we use the notations and normalizations of Manolescu.\\

\noindent {\bf The injection $\phi$ from $\mathcal{G}$ to $\mathcal{H}$.} Fix an oriented rectangular diagram $D$. We construct a bijection between the set $\mathcal{G}$ and a subset of the enhanced Kauffman states of an oriented link diagram. An enhanced Kauffman state of $D$ is a choice of one resolution for each crossing of $D$ (see Figure \ref{kauf}), together with a choice of orientation on every resulting circle, see Figure \ref{kstate} for an example. We call a choice of resolution for each crossing of $D$ a {\it resolution} of $D$. We define $\mathcal{K}$ to be the set of enhanced Kauffman states.
\begin{figure}[!h]
\begin{center}
\includegraphics[height=1cm]{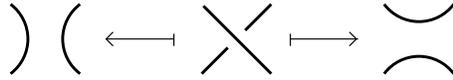}
\caption{Kauffman resolutions}
\label{kauf}
\end{center}
\end{figure}

\begin{figure}[!h]
\begin{center}
\includegraphics{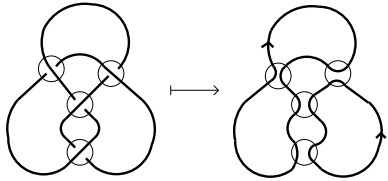}
\caption{Example of an enhanced Kauffman state}
\label{kstate}
\end{center}
\end{figure}

We define $\mathcal{H}$ to be the set of enhanced Kauffman states associated to $D$ such that around each crossing, the arcs coming from the resolutions are oriented as in one of the configurations depicted in Figure \ref{possible}.

\begin{figure}[!h]
\begin{center}
\includegraphics[height=1.4cm]{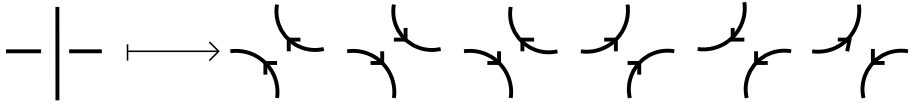}
\caption{ Allowed orientations around a resolution in an enhanced Kauffman state}
\label{possible}
\end{center}
\end{figure}
Notice that the enhanced Kauffman states that are not in $\mathcal{H}$ are those for which at least near one crossing the orientation is as in Figure \ref{Forbidden}.\\

\begin{figure}[!h]
\begin{center}
\includegraphics{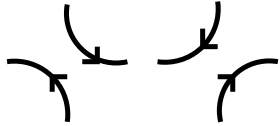}
\caption{Forbidden orientations around a resolution in an enhanced Kauffman state}
\label{Forbidden}
\end{center}
\end{figure}

Given $g=(g_1,\ldots,g_n) \in \mathcal{G}$, we explain how to associate to $g$ an element $\phi(g)\in \mathcal{H}$, see e.g. Figure \ref{exampleg}.
Thinking of $D$ as a set of horizontal and vertical segments, we can subdivide all the segments in $D$ at the crossings of $D$ to obtain a set of smaller segments that we denote by $\seg(D)$.
To an element $g \in \mathcal{G}$ corresponds an orientation of the segments of $\seg(D)$ by the following two rules:
\begin{itemize}
\item{A vertical segment $s$ is oriented upward when the intersection point of $\overline{g}$ nearest to the line containing $s$ is lower than $s$. It is oriented downward otherwise.}
\item{An horizontal segment $s$ is oriented leftward when the intersection point of $\overline{g}$ nearest to the line containing $s$ is at the left of $s$. It is oriented rightward otherwise.}
\end{itemize}

This means that around an intersection point of $\overline{g}$ the orientation of $\seg(D)$ looks like in Figure \ref{bifurcation}. See the Figure \ref{exampleg}, the diagram in the middle for an another example.

\begin{figure}[!h]
\begin{center}
\includegraphics{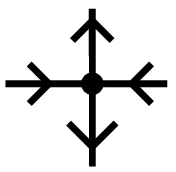}
\caption{Orientation near a $\overline{g_i}$}
\label{bifurcation}
\end{center}
\end{figure}

\begin{figure}[!h]
\begin{center}
\includegraphics{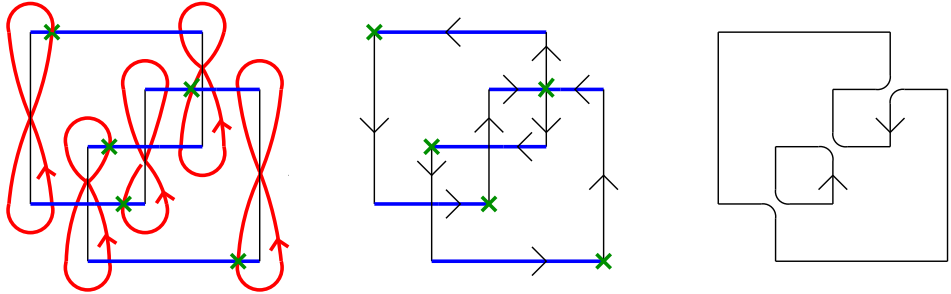}
\caption{An element of $\mathcal{G}$, the orientation it induces on $\seg(D)$ and its corresponding enhanced Kauffman state $\phi(g)$}
\label{exampleg}
\end{center}
\end{figure}
We obtain the enhanced Kauffman state $\phi(g)$ from the orientation on $\seg(D)$ by specifying a resolution of each crossing of $D$. The orientation on the circles of $\phi(g)$ being induced in the obvious way by the orientation on $\seg(D)$.
For a crossing $c$, if there is no intersection point of $g$ near $c$, we resolve $c$ in the only way that is coherent with the orientation on $\seg(D)$ (see Figure
 \ref{4}). If there is an intersection point $\overline{g_i}$ of $\overline{g}$ in $c$, the resolution depends on the position of $g_i$. There are four possible cases and Figure \ref{gi} describes how to resolve in each case. Since the orientation on $\seg(D)$ is coherent with our choices of resolutions, the function $g\longrightarrow \phi(g)$ is well defined.

\begin{figure}[!h]
\begin{center}
\includegraphics{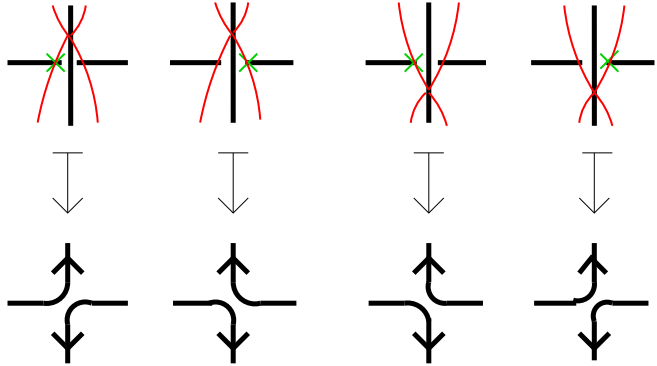}
\caption{Resolution near a $g_i$}
\label{gi}
\end{center}
\end{figure}

\begin{thm}
Let $D$ be an oriented planar rectangular diagram.
The application $\phi : g \mapsto \phi(g)$ defines a bijection between $\mathcal{G}$ and $\mathcal{H}$.
\end{thm}

\begin{proof}

We define a map $\psi$ from $\mathcal{H}$ to $\mathcal{G}$, see Figure \ref{reciproc} for an example. Given an enhanced Kauffman state $h$ in $\mathcal{H}$, consider the orientation induced by $h$ on $\seg(D)$. 
\begin{figure}[!t]
\begin{center}
\includegraphics{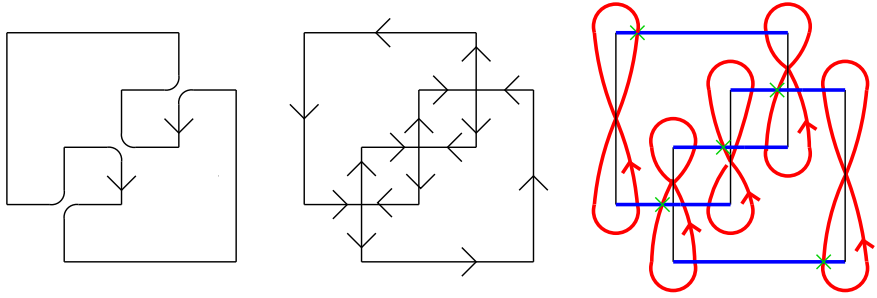}
\caption{Orientation induced on $D$ by an enhanced Kauffman state and associated element of $\mathcal{G}$}
\label{reciproc}
\end{center}
\end{figure}
\begin{figure}[!t]
\begin{center}
\includegraphics{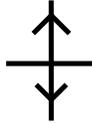}
\caption{Switch of orientation}
\label{switch}
\end{center}
\end{figure}
We claim that on a vertical segment of $D$, there is at most one switch of orientation of the kind shown in Figure \ref{switch}.
This follows from the fact that if there were two, then between them, there would be one switch of the type depicted in Figure \ref{badswitch}, which is impossible since enhanced Kauffman state in $\mathcal{H}$ do not contain resolutions oriented as in Figure \ref{Forbidden}.
\begin{figure}[!h]
\begin{center}
\includegraphics{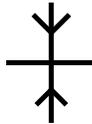}
\caption{Forbidden switch of orientation}
\label{badswitch}
\end{center}
\end{figure}
A similar argument proves that, on each horizontal segment, there is at most one switch of orientation (of the allowed type). It follows that one can associate to an element $h\in \mathcal{H}$, an $n$-tuple  $\overline{\psi(h)}$ of $X$'s, $O$'s or crossings, such that for any $g\in \mathcal{G}$, $\overline{g}=\overline{\psi(h)}$ implies that $g$ and $h$ induce the same orientation on $\seg(D)$. By using Figure \ref{gi}, we replace crossings in $\overline{\psi(h)}$ by elements of $\mathcal{Z}$ according to the kind of resolutions around the crossings and the self-intersections of the figure-eights. We replace punctures in $\overline{\psi(h)}$ by their nearest point in $\mathcal{Z}$. We obtain a $n$-tuple $\psi(h)\in \mathcal{G}$. Since clearly $\phi $ and $\psi$ are inverses, $\phi$ defines a bijection between $\mathcal{G}$ and $\mathcal{H}$.
\end{proof}

\section{Gradings}

\noindent {\bf Khovanov homological grading and quantum grading.}
We introduce two gradings $i$ and $j$ on $\mathcal{H}$. By the previous theorem, they induce two gradings on $\mathcal{G}$. Given an element $h\in\mathcal{H}$, consider the underlying resolution $r(h)$ and define $\overline{i}(h)$ to be the number of resolutions in $r(h)$ of the type depicted in Figure \ref{i}. Then 
$i(h)$=$\overline{i}(h)-n_-$. We define $j(h)=\rot(h)+\overline{i}(h)+n_+-2n_-$. Notice that $J({\bf x})=j({\bf x})$.
\begin{figure}[!h]
\begin{center}
\includegraphics{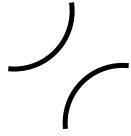}
\caption{Resolution used in the definition of Khovanov homological grading on $\mathcal{H}$}
\label{i}
\end{center}
\end{figure}
As a consequence, we define for any $g\in\mathcal{G}$,
\begin{equation}
j(g)=j(\phi(g))=\rot(\phi(g))+\overline{i}(\phi(g))+n_+-2n_-,
\label{JeRI}
\end{equation}
\begin{equation}
i(g)=i(\phi(g))=\overline{i}(\phi(g))-n_-.\label{IeIbar}
\end{equation}
We express the grading $j(g)$ directly from $g$ described as a set of intersection points. Let us decompose the grading $j$ as a sum of three gradings $j_1$, $j_2$ and $j_3$.\\

We define $j_1(g)$ to be  an algebraic count on the corners of the rectangular diagram. Each corner of the rectangular diagram has a contribution of $+1$ or $-1$. Given $g$, $j_1(g)$ is the sum over all corners of these contributions. For the contributions of each corners, see Figure \ref{j1+} and Figure \ref{j1-}.\\
\begin{figure}[!h]
\begin{center}
\includegraphics{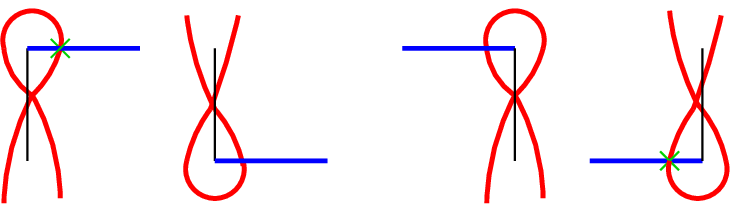}
\caption{Corners with a contribution to $j_1$ of $+1$}
\label{j1+}
\end{center}
\end{figure}
\begin{figure}[!h]
\begin{center}
\includegraphics{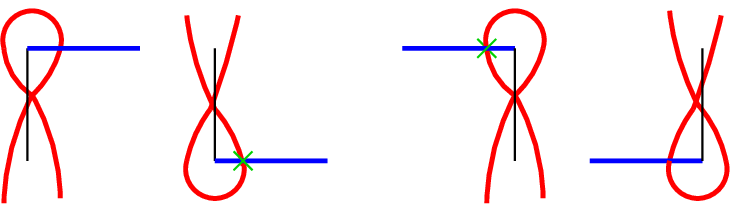}
\caption{Corners with a contribution to $j_1$ of $-1$}
\label{j1-}
\end{center}
\end{figure}

We define $j_2(g)$ to be an algebraic count on the crossings that are near an intersection point of $g$. For $g=(g_1,\ldots,g_n)$, we consider $\overline{g}=(\overline{g_1},\ldots,\overline{g_n})$. If $\overline{g_i}$ is on a crossing, it has a contribution of $+1$ or $-1$ depending on the positions of $g_i$ and of the self-intersection of the corresponding figure-eight, see Figure \ref{j2+} and Figure \ref{j2-}.\\

\begin{figure}[!h]
\begin{center}
\includegraphics{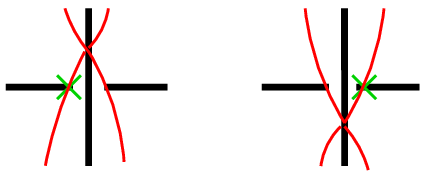}
\caption{Crossings with a contribution to $j_2$ of $+1$}
\label{j2+}
\end{center}
\end{figure}
\begin{figure}[!h]
\begin{center}
\includegraphics{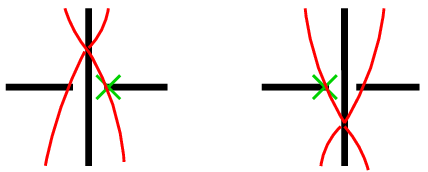}
\caption{Crossings with a contribution to $j_2$ of $-1$}
\label{j2-}
\end{center}
\end{figure}

We define $j_3(g)$ to be an algebraic count on the crossings that do not have a point of $g$ nearby. For $g=(g_1,\ldots,g_n)$, consider $\overline{g}=(\overline{g_1},\ldots,\overline{g_n})$ and consider the crossings where there is no $\overline{g_i}$. Each of these crossings has a contribution of $+1$ or $-1$ depending on the position of the $g_i$ that is on the same figure-eight and on the position of the $g_j$ that is on the same horizontal segment, see Figure \ref{j3+} and Figure \ref{j3-}. Notice that $j_3$ does not depend on the position of the self-intersection.\\

 From previous definitions, we deduce that, for all $g\in \mathcal{G}$,
$$\rot(\phi(g))=\frac{j_1(g)}{4}+\frac{j_2(g)}{2} \mbox{ and }$$
\begin{equation}
i(g)=i(\phi(g))=\frac{j_2(g)}{2}+\frac{j_3(g)}{2}+\frac{n_+-n_-}{2}.\label{combDef0}
\end{equation}
Hence,
\begin{equation}
j(g)=\frac{j_1(g)}{4}+j_2(g)+\frac{j_3(g)}{2}+\frac{3}{2}(n_+-n_-), \mbox{ for all } g\in \mathcal{G}.\label{combDefJ}
\end{equation}
\begin{figure}[!h]
\begin{center}
\includegraphics{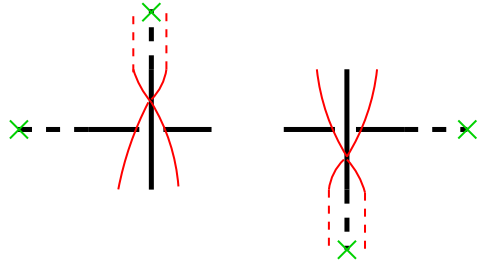}
\caption{Crossings with a contribution to $j_3$ of $+1$}
\label{j3+}
\end{center}
\end{figure}
\begin{figure}[!h]
\begin{center}
\includegraphics{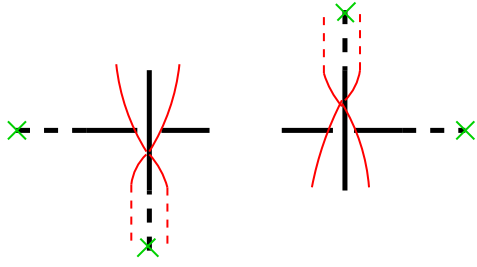}
\caption{Crossings with a contribution to $j_3$ of $-1$}
\label{j3-}
\end{center}
\end{figure}

\noindent {\bf Relations between gradings.} We prove that the gradings verify $P=i-j$ and $J=j$.

\begin{thm}\label{P=i-j}
Let $D$ be an oriented planar rectangular diagram.
For all $g\in\mathcal{G}$,
$$P(g)=i(g)-j(g).$$
\end{thm}

\begin{proof}
We denote by $\mathcal{G}'$  the set of unordered $n$-tuples of intersection points between horizontal segments and vertical figure-eights such that each (vertical) figure-eight 
contains exactly one point. Observe that $\mathcal{G}$ is a subset of $\mathcal{G}'$. Moreover, the gradings $j_1$, $j_2$ and $P$ on $\mathcal{G}$ extend in a natural way to $\mathcal{G}'$. Since
$$j({\bf x})=\rot(D)+w(D)\mbox{ , } i({\bf x})=0 \mbox{ and } P({\bf x})=-\rot(D)-w(D),$$
we have $$P({\bf x})=i({\bf x})-j({\bf x}).$$

\begin{figure}[!h]
\begin{center}
\includegraphics{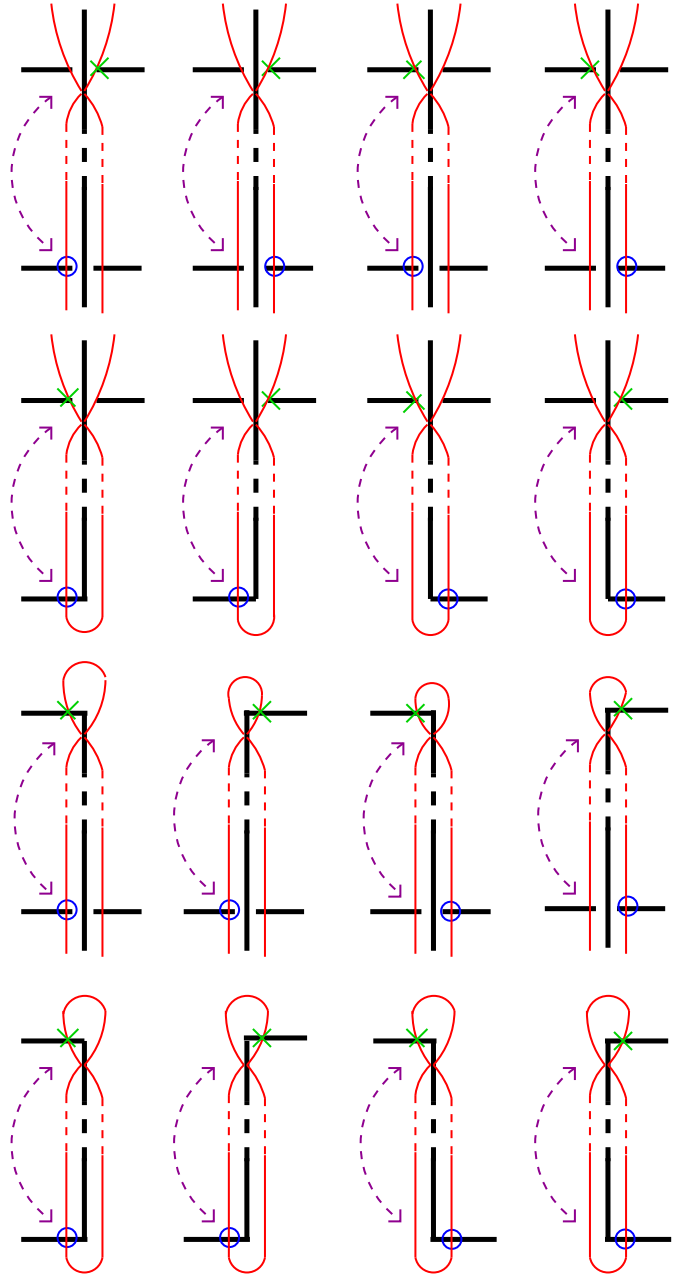}
\caption{Elementary moves}
\label{elem}
\end{center}
\end{figure}

We prove that $P$ and $i-j$ are equal as relative gradings. Notice that one can go from an element of $\mathcal{G}'$ to another by a sequence of elementary moves changing the position of just one intersection point on one figure-eight. It remains to check that $P$ and $i-j$ change by the same amount when such an elementary move is performed. The different cases of elementary moves are presented in Figure \ref{elem}. Moreover, combining Equations (\ref{combDef0}) and (\ref{combDefJ}), we obtain
$$j(g)=\frac{j_1(g)}{4}+\frac{j_2(g)}{2}+i(g)+w(D),\mbox{ for all } g\in \mathcal{G}.$$
Thus, we need to prove that $\frac{j_1}{4}+\frac{j_2}{2}$ and $-P$ vary by the same amount when  an elementary move is performed. This can be checked directly from the pictures in Figure \ref{elem}. Notice that there are a priori 16 other cases to check corresponding to changes of positions of the self-intersections of the figure-eights in the twelve first elementary moves depicted in Figure \ref{elem}, but these cases follow from the following observation: A change of position of the self-intersection of the figure-eight together with a change of position of the generator by switching side if it sits between the old and the new self-intersections changes neither $\frac{j_1}{4}+\frac{j_2}{2}$ nor $-P$, see Figure \ref{flip}.

\begin{figure}[!h]
\begin{center}
\includegraphics{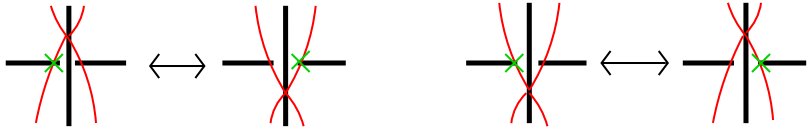}
\caption{Change of position of a self-intersection point of a figure-eight}
\label{flip}
\end{center}
\end{figure}

\end{proof}

\begin{cor}
For any $g \in \mathcal{G}$, we have the equality $$P(g)=-\rot(g)-n_++n_-.$$
\end{cor}
\begin{proof}
Combining Equations (\ref{JeRI}) and (\ref{IeIbar}) with the previous theorem, we obtain the desired equality.
\end{proof}
\begin{thm}\label{J=j}
Let $D$ be an oriented planar rectangular diagram, for all $g\in\mathcal{G}$,
$$J(g)=j(g).\label{theoJ}$$
\end{thm}
\begin{proof}
Since
$$j({\bf x})=\rot(D)+w(D)= J({\bf x}),$$
it is sufficient to prove that $j$ and $J$ are equal as relative gradings.
\begin{figure}[!h]
\begin{center}
\includegraphics[scale=0.4]{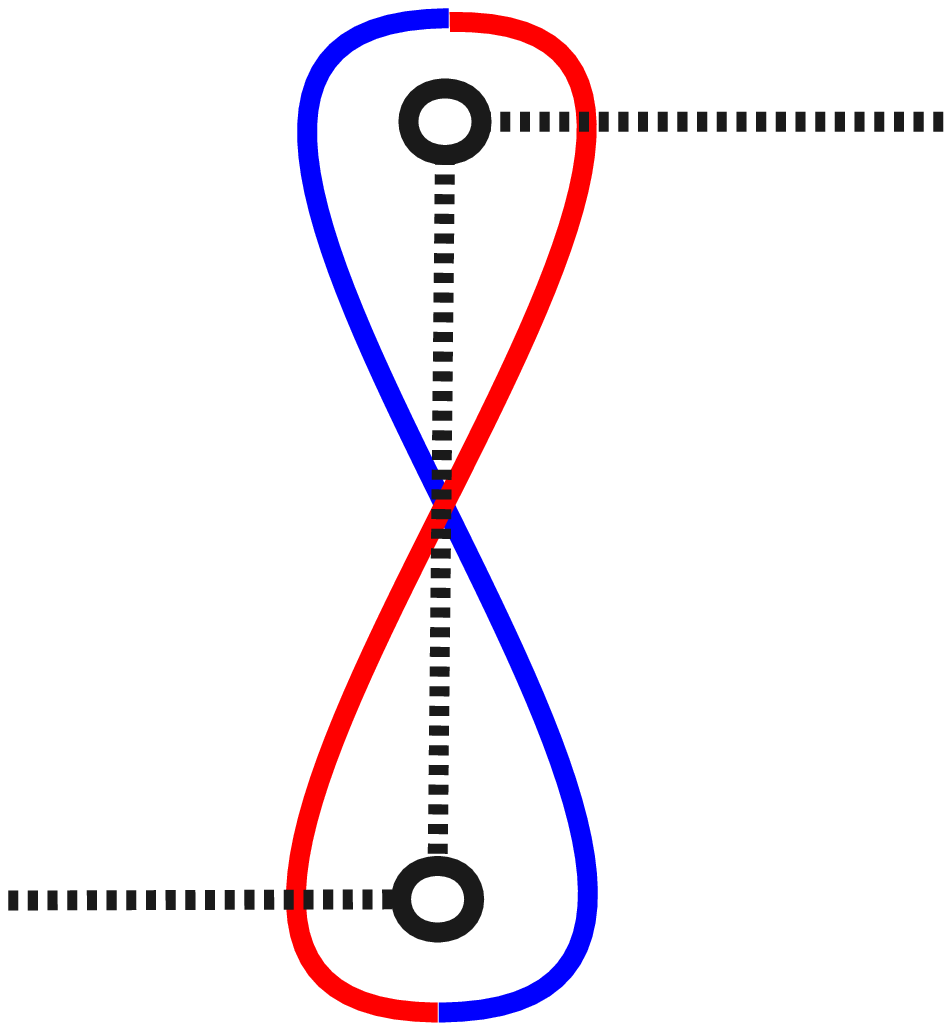}
\caption{The two parts of the figure-eight}
\label{DividedFigEight}
\end{center}
\end{figure}

For $A$ and $B$ two sets of intersection points, we define $\mathcal{I}^-(A,B)$ as the 
number of pairs $(a_1,a_2)\in A$ and $(b_1,b_2)\in B$ such that $a_1<b_1$ and $a_2>b_2$. Given $g=(g_1,\ldots,g_n)\in \mathcal{G}$, we define $Q_{far}(g)$ as follows: 
$$Q_{far}(g)=\mathcal{I}(g,\mathbb{X}\cup \mathbb{O})+\mathcal{I}(\mathbb{X}\cup \mathbb{O},g)-\mathcal{I}^-(g,\mathbb{X}\cup \mathbb{O})-\mathcal{I}^-(\mathbb{X}\cup \mathbb{O},g).$$

We divide any figure-eight in a positive and a negative part by cutting it at its lowest and heighest points, Figure \ref{DividedFigEight}. We define $Q_{\loc}(g)$ as the number of intersection points of $g$ on the positive side of the figure-eight (in blue) minus the number of points on the negative side of the figure-eight (in red). We observe that \begin{equation}2P=_{\rel}Q_{\loc} \label{PeQ}.\end{equation} (The notation $=_{\rel}$ means that the two sides of the equation are equal up to an additive constant.)

The relative grading $Q$ between two generators $g$ and $h$ is defined as the total winding number of a set of closed paths around the punctures. Let us call $L$ the set of vertical and horizontal half-lines originating from punctures. The winding number of a path can be computed by counting algebraically (with signs) the intersection points of the path with $L$ and dividing by 4. Using this alternative definition of the winding number and examining separetely the contribution of the pieces of the path on horizontal lines and on figure-eights, we check that $4Q(g)-4Q(h)=2Q_{\loc}(g)-2Q_{\loc}(h)+Q_{far}(g)-Q_{far}(h)$, in other words, that, as relative gradings, $4Q=_{\rel}2Q_{\loc}+Q_{far}$. 

Let us call {\it good} a line either vertical or horizontal that meets two punctures. We call intersection points of good lines {\it good points}. 
\begin{claim}
We have the equality of relative gradings:
\begin{equation}
4T-j_3+\frac{j_1}2=_{\rel}Q_{far}.
\label{lemDC}
\end{equation}
\end{claim}
\begin{proof}
For each pair of one horizontal good line and one vertical good line, we examine the contributions of all pairs of one puncture or intersection point on the vertical line and one puncture or intersection point on the horizontal line to $4T$, $j_3$, $\frac{j_1}2$ and $Q_{far}$. We consider different cases according to where the two good lines intersect. We check that for each pair of good lines, the total contribution to Equation (\ref{lemDC}) from points on the two lines are independent of the positions of the intersection points of $g$. For example, if the two good lines intersect on a crossing of the rectangular diagram, only $4T$ and $j_3$ are affected and the contribution of the intersection points to $4T$ and $j_3$ cancel. Since each puncture or intersection point appears once on a vertical good line and once on an horizontal good line, taking the sum over all possible pairs of good lines gives twice Equation \ref{lemDC}.
\end{proof}
Using Equation (\ref{lemDC}) divided by two and the definition of $j$ by Equation (\ref{combDefJ}), we get
$$j=_{\rel}\frac{j_1}4+j_2+2T-\frac{Q_{far}}2+\frac{j_1}4.$$
Using Equation (\ref{PeQ}) and the formulation of $P$ in terms of $j_1$ and $j_2$, we have
$$-Q_{\loc}=_{\rel}\frac{j_1}2+j_2.$$
Combining the last two equations gives
$$j=_{\rel}2T-\frac{Q_{far}}2-Q_{\loc}.$$
Since $\frac{Q_{far}}2+Q_{\loc}=_{\rel}2Q$ and $J=_{\rel}2(T-Q)$,
we deduce $j=_{\rel}J$ from which $j=J$ follows.
\end{proof}

\section{Relation with Khovanov homology}

In \cite{M}, Manolescu noticed that, in the case of the trefoil, a free abelian group whose generators are labelled by Bigelow's generators and are graded according to the grading of their label, has ranks in the different gradings compatible with having as homology the Khovanov homology. We prove that this observation holds for any link. For this purpose, we start with the Khovanov chain complex, i.e the chain complex whose generators are enhanced Kauffman states $\mathcal{K}$ and cancel all generators lying in $\mathcal{K}\setminus\mathcal{H}$. We end up with a chain complex homotopic to the original one, with a set of generators in one-to-one correspondance with $\mathcal{G}$. This reduction can be done canonically over
$\mathbb{Q}$. Over $\mathbb{Z}$ the reduction seems to depend on some arbitrary choices.\\

Our reduction over $\mathbb{Q}$ to a smaller complex is {\it canonical} in the following sense. Two oriented rectangular diagrams that are sent to each other by  diffeomorphisms of the plane that send horizontal segments to horizontal segments have isomorphic chain complexes. This implies that the chain complex generated by Bigelow's intersection points is computable from the ambient isotopy type of the flattenend braid diagram and is as such a candidate for geometric interpretation.
\begin{proof}{[of Theorem \ref{main}]}
\noindent We review a few facts about Khovanov homology (for precise definitions and more \cite{Kh, Vi}). The Khovanov chain complex is a bigraded complex
$$C_{\Kh}=\bigoplus_{i,j} C_{\Kh}^{i,j}$$
with generators associated to enhanced Kauffman states
$$C_{\Kh}^{i,j}=\bigoplus_{s \in \mathcal{K}, i(s)=i, j(s)=j} \mathbb{Z}\cdot s.$$

Observe that $C_{\Kh}^i=\bigoplus_{j\in \mathbb{Z}}C_{\Kh}^{i,j}$ is generated by all enhanced Kauffman states obtained by orienting the circles in a resolution $r$ such that $i(r)=i$. Hence, as usual, $C_{\Kh}^i$ can be seen as $\bigoplus_{r, i(r)=i} V^{\otimes k(r)}\{i+n_+-n_-\}$ where $V$ is the two dimensional graded  $\mathbb{Z}$-module spanned by $1$ and $x$ with $j(1)=1$ and $j(x)=-1$, $k(r)$ is the number of circles in the resolution $r$ and $\{\cdot\}$ is the shift operator in homological grading. One can identify an element in $V(r)=V^{\otimes k(r)}$ with a choice of orientations of the $k(r)$ circles. A circle oriented counterclockwise corresponds to a $1$ and a circle oriented clockwise corresponds to an $x$. Given a diagram $D$, with $k$ crossings there are $2^k$ resolutions of $D$. One can see these resolutions as lying on the vertices of an hypercube of dimension $k$. Hence, the vector spaces $V(r)$ are indexed by the vertices of this hypercube. Similarly, each enhanced Kauffman state sits on a vertex of the hypercube. Two enhanced Kauffman states are connected by the differential of Khovanov homology if and only if they are on both ends of an edge of the hypercube and look around a crossing like one of the 18 pairs in Figure \ref{diff}. Locally, all possible non-zero differentials are depicted in Figure \ref{diff}.
 Naturally, one would have to introduce signs to fully specify the differential \cite{Kh, Vi} .\\

We define a new grading $R$ on the Khovanov chain complex as follows. Given an oriented link diagram $D\in \mathbb{R}^2$, consider the underlying  oriented 4-valent graph $\Gamma$ in which each crossing of $D$ is replaced by a 4-valent vertex. Choose a point in each connected component of $\mathbb{R}^2\setminus\Gamma$. This produces a family of points $(x_i)_{i\in I}$, where $I$ is a finite set. For an enhanced Kauffman state $s$, we define
 $R(s)$ to be the winding number of the oriented circles of $s$ around the $x_i$'s.
 In Figure \ref{diff}, the blue cross and the green cross correspond to $x_i$'s.\\

The six differentials in the first column of Figure \ref{diff} are connecting generators belonging to $\mathcal{K}\setminus\mathcal{H}$. These differentials respect the grading $R$. In addition, the remaining twelve other types of differentials stricly decrease $R$. Therefore, the increasing filtration associated to the grading $R$ is respected by the differential of the Khovanov complex. As a consequence, the set of elements of $\mathcal{K}\setminus\mathcal{H}$ together with the part of the differential that respects the grading $R$ is a chain complex. Moreover, this chain complex is homotopic to zero. This follows from the fact that this chain complex is a shifted direct sum of hypercube chain complexes, where each hypercube chain complex is obtained by the usual procedure of flattening an hypercube \cite[p.17]{OS2}  and replacing every vertex by a copy of $\mathbb{Z}$ and every arrow by $+\Id$ or $-\Id$. Each of those hypercube chain complexes is clearly homotopic to zero.\\

\begin{figure}[!h]
\begin{center}
\includegraphics{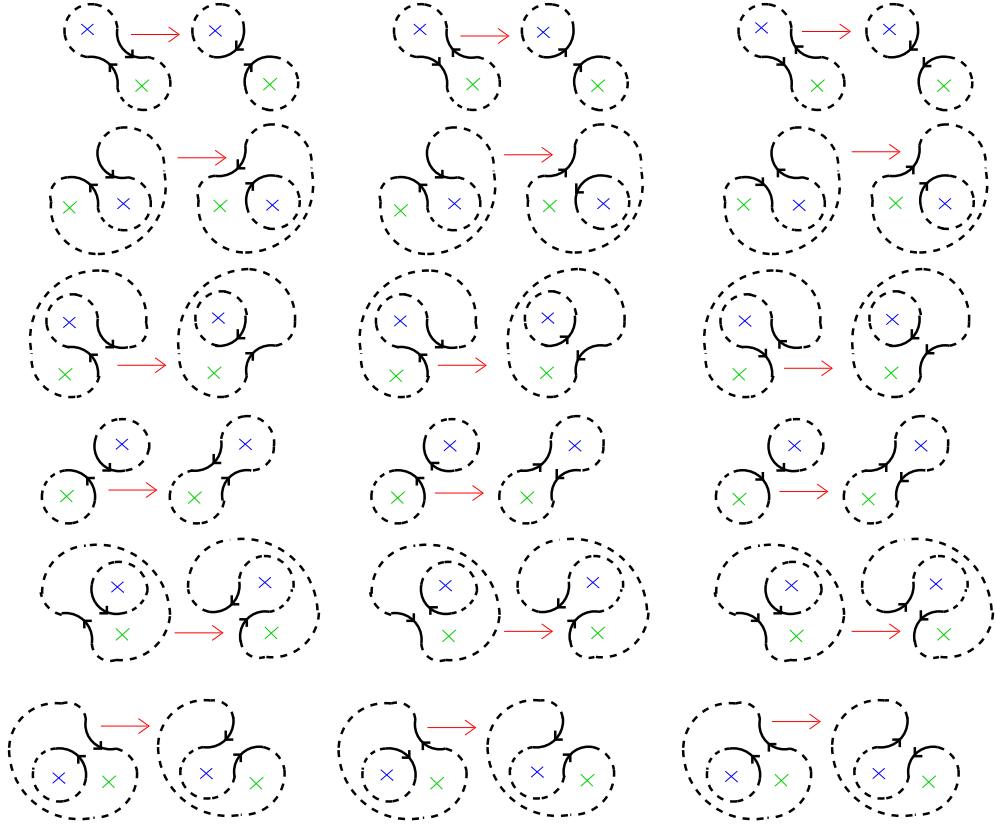}
\caption{Non-zero differentials in Khovanov complex}
\label{diff}
\end{center}
\end{figure}

We state a lemma in homological algebra generalizing Gaussian elimination. The lemma is proved in the Appendix.
\begin{lem}\label{Gauss}
Given a chain complex $C=\bigoplus_{i\in\mathbb{Z}}C^i$, $C\stackrel{\delta}{\longrightarrow}C$, $\delta$ of degree $+1$ such that $C=A\oplus B$ as graded abelian group and $\delta=\left(\begin{array}{cc} a & d\\ c & b\end{array}\right)$ with $ A\stackrel{a}{\longrightarrow}A$, $ B\stackrel{d}{\longrightarrow}A$, $ A\stackrel{c}{\longrightarrow}B$ and $B\stackrel{b}{\longrightarrow}B$. Suppose that $(B,b)$ is a chain complex homotopic to zero. Then, for $h:B{\longrightarrow}B$ of degree $-1$ such that $-\Id=hb+bh$, the chain complex $(C,\delta)$ is homotopic to $(A,a+dhc)$.
\end{lem}

Theorem \ref{main} follows from Theorem \ref{P=i-j}, Theorem \ref{J=j} and Lemma \ref{Gauss}.
\end{proof}

\begin{cor}[Bigelow \cite{Bi}]
Given an oriented rectangular diagram $D$, 
$$V(D)(q)=\sum_{g \in \mathcal{G}} {(-1)}^{P(g)} q^{J(g)}$$
is the Jones polynomial.
\end{cor}

\begin{cor}\label{cormain}
There exists a differential $\delta$ on the $\mathbb{Q}$-vector space $B$ generated by Bigelow's generators $\mathcal{G}$, that respects $J$, increases $P$ by 1 and such that the homology of the chain complex $(B,\delta)$ is the Khovanov homology. Moreover, over $\mathbb{Q}$, this differential is canonical.
\end{cor}

\begin{proof}
The first part of the statement follows from Theorem \ref{main}.
The second part follows from the fact that, over $\mathbb{Q}$, there is a canonical homotopy for Lemma \ref{Gauss}. We construct this homotopy.
The set of elements of $\mathcal{K}\setminus\mathcal{H}$ together with the part of the differential that leaves the grading $R$ invariant is a chain complex denoted by $C$. 
This chain complex is a shifted direct sum of hypercube chain complexes, where each hypercube chain complex is obtained by the usual procedure of flattening an hypercube and, in this case, replacing every vertex by a copy of $\mathbb{Q}$ and every arrow by $+\Id$ or $-\Id$. As explained before, those hypercubes are null homotopic.  In particular, given a decomposition of an hypercube of dimension $n$ into two hypercubes of dimension $n-1$, one can take as homotopy minus the inverse of the differentials between the two hypercubes of dimension $n-1$. Since there are $n$ ways to split an hypercube of dimension $n$ into two hypercubes of dimension $n-1$, there are $n$ homotopies of the type described above. We denote them by $h_1,h_2,\ldots,h_n$. Each of them comes from a choice of splitting. We consider $h=\frac{1}{n}\sum_{i=1}^n h_i$ ; it is a homotopy to zero.
Summing over all hypercubes composing the chain complex $C$, the average homotopies described above, one obtains a homotopy to zero for the whole complex $C$.
Hence, over $\mathbb{Q}$, there is a canonical choice of homotopy for the application of Lemma \ref{Gauss} and therefore, a canonically defined differential on the set of Bigelow's generators.
\end{proof}

Theorem \ref{main} and Corollary \ref{cormain} remain true for the odd Khovanov homology \cite{ORS}. More precisely, the proof only depends on the fact that one can endow enhanced Kauffman's states with a differential that respects the increasing filtration associated to the grading $R$ and such that the part of the differential that respects the grading $R$ is connecting elements in $\mathcal{K}\setminus\mathcal{H}$.  The differential constructed in \cite{ORS} is up to signs the original Khovanov differential. Hence, we also have the following theorem:
\begin{thm}\label{odd}
There exists a differential $\delta$ on the free abelian group $B$ generated by Bigelow's generators $\mathcal{G}$, that respects $J$, increases $P$ by 1 and such that the homology of the chain complex $(B,\delta)$ is the odd Khovanov homology. Moreover, over $\mathbb{Q}$, this differential is canonical.
\end{thm}

\section*{Appendix}

We give the proof of Lemma \ref{Gauss}.

\begin{proof}
We begin by proving that $(A,a+dhc)$ is a chain complex. Notice that $\delta^2=0$ and $b^2=0$ imply
\begin{eqnarray}\label{fund}  cd = & 0. \end{eqnarray}
Moreover, $\delta^2=0$ also implies 
\begin{eqnarray}  a^2+dc, & = & 0,\label{ea}\\
 ad+db & = & 0, \label{eb}\\
\end{eqnarray}
and
\begin{eqnarray}
 ca+bc & = & 0. \label{ec} \end{eqnarray}
Hence, 
\begin{eqnarray*}
(a+dhc)(a+dhc) & = & a^2+ dhca+adhc+dhcdhc\\
                                                                       & = & a^2+ dhca+adhc\\
                                                                       & = & -dc- dhbc-dbhc\\
                                                                       & = & -d(\Id+hb+bh)c\\
                                 &=& 0.
\end{eqnarray*}

\noindent We define a chain map  $f$ from $(C,\delta)$ to $(A,a+dhc)$ by the formula
$$f=\left(\begin{array}{cc} 1 & dh\end{array}\right)$$
and a chain map $g$ from $(A,a+dhc)$ to $(C,\delta)$ by the formula
$$g=\left(\begin{array}{c} 1 \\ hc \end{array}\right).$$
It is straightforward to check that $f\delta=(a+dhc)f$ and $\delta g =g(a+dhc)$ using (\ref{fund}), (\ref{ec}), (\ref{eb}) and $-\Id=hb+bh$.
We define a homotopy $H=\left(\begin{array}{cc} 0 & 0\\ 0 & h\end{array}\right)$ going from $C$ to $C$.
If follows from (\ref{fund}) and $-\Id=hb+bh$ that $gf-\Id=H\delta+\delta H$.
We define $H'$ as follows: $H'=dh^3c$.
Let us check that $gf-\Id=H'(a+dhc)+(a+dhc)H'$. We have
\begin{eqnarray}\label{zut}
gf-\Id  =  \Id+dh^2c-\Id =  dh^2c.
\end{eqnarray}
Moreover, we have
\begin{eqnarray*}
       dh^2c   & = & -dhbh^2c-dh^2bhc\\
               & = & 2dh^2c +dbh^3c+dh^3bc\\
               & = & 2dh^2c -adh^3c-dh^3ca,
\end{eqnarray*}
from which we deduce
 \begin{eqnarray} \label{zut2}
 dh^2c   & = & dh^3ca + adh^3c.
\end{eqnarray}
Hence, combining (\ref{zut}) and (\ref{zut2}), we obtain
$$gf-\Id = dh^3ca + adh^3c.$$
Moreover, \begin{eqnarray*}
 H'(a+dhc)+(a+dhc)H'&  = & dh^3c(a+dhc)+(a+dhc)dh^3c \\
& = & dh^3ca+adh^3c.
\end{eqnarray*}
\end{proof}

Jean-Marie Droz, \emph{Institut f\"ur Mathematik, Universit\"at Z\"urich, Winterthurerstrasse 190, CH-8057 Z\"urich, Switzerland.}\\
E-mail address: jdroz@math.unizh.ch\\

\noindent Emmanuel Wagner, \emph{Department of Mathematics, University of Aarhus, DK-8000, Denmark.}\\
E-mail address: wagner@imf.au.dk

\end{document}